\documentclass[a4paper,12pt]{amsart}
\usepackage{graphicx,amsmath,amsfonts,amssymb,amsthm,mathrsfs,color,url}

\newcommand{\ZZ}{{\mathbb Z}}
\newcommand{\CC}{{\mathbb C}}
\newcommand{\QQ}{{\mathbb Q}}

\def\Fq2{{\mathbb F}_{q^2}}
\def\Fp2{{\mathbb F}_{p^2}}

\def\Gal{{\rm Gal}}
\def\End{{\rm End}}

\def\cC{{\mathcal C}}

\def\cX{{\mathcal X}}

\def\cY{{\mathcal Y}}
\def\cD{{\mathcal D}}
\numberwithin{equation}{section}
\theoremstyle{plain}
\newtheorem{thm}{Theorem}
\numberwithin{thm}{section}
\newtheorem{lem}[thm]{Lemma}

\newtheorem{pro}[thm]{Proposition}

\newtheorem{remark}[thm]{Remark}

\def\blfootnote{\xdef\@thefnmark{}\@footnotetext}
\begin{document}
\title{Explicit Families of Hyperelliptic Curves with CM Jacobians}

\keywords{Complex Multipications, hyperelliptic curves, Jacobians, abelian varieties,
 Chebyshev polynomials.}
 \subjclass[2000]{11G15, 14K22}
\author{Saeed Tafazolian and Jaap Top
   }
\date{}
\address{Departamento de Matem\'{a}tica - Instituto de Matem\'{a}tica, Estat\'{i}stica e Computação Cient\'{i}fica
(IMECC) - Universidade Estadual de Campinas (UNICAMP), Rua S\'{e}rgio Buarque de Holanda, 651, Cidade Universit\'{a}ria,  Zeferino Vaz, Campinas, SP 13083-859, Brazil}
\address{Bernoulli Institute for Mathematics, Computer Science, and Artificial Intelligence,\\
Nijenborgh~9\\9747 AG Groningen\\ the Netherlands}
\email{saeed@unicamp.br}
\email{j.top@rug.nl}



\begin{abstract}
We construct explicit families of hyperelliptic curves over $\QQ$ whose Jacobians admit complex multiplication (CM). Each curve in these families is defined by
\[
v^2 = (u+2)\,\varphi_d(u), \quad d = 2^e \text{ or } d=p \geq 3 \text{ prime},
\] 
where $\varphi_d(x)$ is the Chebyshev polynomial of degree $d$. We prove that the Jacobians are simple and determine the associated CM-fields explicitly. Our approach exploits the interplay between Chebyshev polynomials and Galois coverings, providing concrete examples of abelian varieties with CM and explicit criteria for their construction.
\end{abstract}

\maketitle

\section{Introduction}

Abelian varieties with complex multiplication (CM) occupy a central place in number theory and algebraic geometry. There are several excellent introductions to this subject,
such as \cite{ST}, \cite[Section~13.3]{BL}, \cite{Milne}, \cite{Lang}, An abelian variety $A$ over an algebraically closed field is said to admit \emph{complex multiplication (CM)} if there exists a number field $K$ of degree $2\dim(A)$ over $\QQ$ such that
\[
    K \subseteq \textrm{End}^0(A)=\textrm{End}(A)\otimes_\ZZ\QQ.
\]
If $A$ admits CM, then moreover this field $K$ can be chosen
to be a CM-field, which means it is a totally imaginary
quadratic extension of a totally real field. We then say
that $A$ has CM by the CM-field $K$. Throughout the
present note we will work in characteristic zero; for
an idea of the in various regards more subtle theory of
CM abelian varieties over fields of positive characteristic,
see, e.g., \cite{Oort}.

If $C$ is a smooth projective curve, one says that $C$ admits complex multiplication if its Jacobian $J(C)$ does.  Abelian varieties with CM have remarkable arithmetic properties: they can be defined over number fields, one has a good understanding of the Newton polygon of their reductions in
positive characteristic and even of the associated zeta function, and more.

Despite their importance, explicit constructions of curves whose Jacobians admit complex multiplication are relatively rare. 
Explicit constructions of such curves often use either the
map $\ZZ[\textrm{Aut}(C)]\to\textrm{End}(J(C))$ or the
fact that all principally polarized abelian varieties
of dimension $\leq 3$ with an indecomposable polarization
are Jacobians of curves. Moreover, if the curve $C$ admits
CM and $C\to C'$ is a finite morphism to another curve $C'$, then $C'$ admits CM as well. For a sample of papers
exploiting any
of these methods, see \cite{Ion}, \cite{BILV}, \cite{TTV}, \cite{CLR}, \cite{CMKT} and \cite[Theorem~5.2.10]{van}).

In this paper, we provide new explicit hyperelliptic curves 
defined over $\QQ$ whose Jacobians admit CM, and we
describe the corresponding CM-fields and CM types. Specifically, the curves defined using an equation
\[
v^2 = (u+2)\,\varphi_d(u),
\]
where $\varphi_d(x)$ is a Chebyshev polynomial of degree $d=2^e\geq 2$ turn out to provide such examples. 
We determine the corresponding CM-fields (which turn out to be
cyclic extensions of $\QQ$), prove that the Jacobians are 
simple, and describe explicitly the CM types involved. 

The main result of this paper is the following (the notation
$\zeta_n$ for a primitive $n$-th root of unity is used).

\begin{thm}\label{mainthm}

For $d=2^e$ resp. for $d=p$ an odd prime, the hyperelliptic curve
\[
\cC_d\colon v^2 = (u+2)\,\varphi_d(u)
\]
has a simple Jacobian, and admits CM with CM-field  generated by $\zeta_{4d} - \zeta_{4d}^{-1}$ (resp. $\zeta_p$). 
\end{thm}

\section{Preliminaries}

In this section we collect basic definitions and results that will be used throughout the paper.  

\subsection{Abelian varieties and Jacobians}

An \emph{abelian variety} $A$ over a field $k$ is a complete, connected algebraic group. The \emph{Jacobian} $\textrm{Jac}(C)$ of a smooth projective curve $C$ of genus $g$ is a $g$-dimensional abelian variety parametrizing divisors of degree $0$ on $C$ modulo linear equivalence.

An abelian variety $A$ is said to be \emph{simple} if it contains no proper nontrivial abelian subvarieties. 
A well-known classical result (see, e.g., \cite[\S12]{CorSil})
is the following.

\begin{thm}[Poincar\'e Complete Reducibility]
Let $A$ be an abelian variety over a field $k$. 
\begin{enumerate}
    \item For any abelian subvariety $B \subseteq A$ there exists an abelian subvariety 
    $C \subseteq A$ such that the map
    \[
    B \times C \;\longrightarrow\; A, \qquad (b,c) \mapsto b+c
    \]
    is an isogeny.
    \item Consequently, every abelian variety is isogenous to a product of simple 
    abelian varieties:
    \[
    A \;\sim\; A_1^{n_1} \times A_2^{n_2} \times \cdots \times A_r^{n_r},
    \]
    where the $A_i$ are simple and pairwise not isogenous.
    \item Moreover, this decomposition is unique up to isogeny and reordering 
    of the factors.
\end{enumerate}
\end{thm}

 \subsection{Complex multiplication and CM type}

Let $A$ be an abelian variety of dimension $g$ defined over a subfield $k$ of $\CC$. 
We assume that $A$ has \emph{complex multiplication} (CM)  by
the CM-field $K$, i.e.,  
\[
K \hookrightarrow \End^{0}(A) := \End(A)\otimes_{\ZZ}\QQ
\]
and $K$ has degree $2g$ over $\QQ$.

The action of $K$ on the tangent space $T_{0}A \otimes_{k} \CC$ at the origin
decomposes as a direct sum of $g$ one-dimensional eigenspaces. 
This decomposition is described by a set of embeddings
\[
\Phi \subseteq \mbox{Hom}(K,\CC), \qquad |\Phi| = g,
\]
called the \emph{CM type} of $A$.  Concretely, $\Phi=\{\sigma_1,\dots,\sigma_g\}$ is a 
choice of one embedding from each pair $\{\sigma,\overline{\sigma}\}$ of complex conjugate 
embeddings of $K$ into $\CC$, and the action of $K$ on 
$T_{0}A\otimes_k\CC$ is given via~$\Phi$.  

A CM type $\Phi$ is called \emph{primitive} if it is not induced from a CM type of a proper CM subfield of $K$. 
Here ``induced from the CM subfield $L\subset K$'' means:
there is a CM type $\Phi_L\subseteq \text{Hom}(L,\CC)$
such that $\Phi=\{\sigma\in \text{Hom}(K,\CC)\;:\; \sigma_{\restriction_L}\in \Phi_L\}$.

The following result is well known (see \cite[\S8]{ST}).

\begin{thm}[Primitive CM type implies simplicity]\label{thm:primitive-simple}
Let $A$ be an abelian variety over $\CC$ with CM by a CM-field $K$, and assume that the associated CM type is primitive. Then $A$ is simple.
\end{thm}

\subsection{Geometric interpretation via tangent spaces}

Let $C$ be a smooth curve with Jacobian $\textrm{Jac}(C)$ and $\tau \in \mbox{Aut}(C)$ an involution. 
The tangent space at the origin of $\textrm{Jac}(C)$ can be identified 
with the space of regular $1$-forms on $C$, i.e.,
\[
T_0 \textrm{Jac}(C) \simeq H^0(C, \Omega^1_C).
\]
If an endomorphism of $\textrm{Jac}(C)$ preserves the subspace of differentials invariant under $\tau$, then it induces a well-defined endomorphism on the quotient Jacobian $\textrm{Jac}(C/\langle \tau \rangle)$. This geometric fact underlies many constructions of CM Jacobians: 
one studies the restriction of a global endomorphism induced
from some element in $\ZZ[\textrm{Aut}(C)]$ by checking its action on the tangent space which in fact boils down to
its action on the space of regular $1$-forms on $C$.
This reduces the problem to linear algebra on the space of differentials. For example, the
CM type of certain CM Jacobians $\textrm{Jac}(C/\langle\tau\rangle)$
is determined in this way; concrete examples illustrating this approach are
presented in \cite[Proposition~4]{TTV}, \cite[Theorem~5.2.10]{van}, and \cite[Lemmas~4.2.1, 4.2.5]{TTCM}.

 \subsection{Some abelian CM-fields}
Here some facts and properties of the
CM-fields that play a role in the current text, are collected.

For any $d\in\ZZ_{\geq 1}$ the $d$-th cyclotomic field $\QQ(\zeta_d)\subset\CC$, 
with $\zeta_d=e^{2\pi i/d}\in\CC$,
is a Galois extension of $\QQ$. Its
Galois group is identified with
$(\ZZ/d\ZZ)^\times$ where
$\bar{n}\in (\ZZ/d\ZZ)^\times$
corresponds to the automorphism $\sigma_n$ defined by $\zeta_d\mapsto \zeta_d^n$.
In particular $\overline{-1}$ corresponds
to complex conjugation on $\QQ(\zeta_d)$.
The standard Galois correspondence in this situation yields
a bijection between the subfields $K\subseteq \QQ(\zeta_d)$
and the subgroups $H\subseteq (\ZZ/d\ZZ)^\times$,
explicitly
\[K\mapsto H:=\{\sigma_n\;:\;\sigma_n(k)=k\;\textrm{for}\;k\in K\}\]
and 
\[H\mapsto \QQ(\zeta_d)^H:=\{\alpha\in\QQ(\zeta_d)\;:\;
\sigma(\alpha)=\alpha\;\textrm{for}\;\sigma\in H\}.\]
If here subgroup $H$ corresponds to subfield $K$, then
$K$ is Galois over $\QQ$ with Galois group $(\ZZ/d\ZZ)^\times/H$. As a consequence, $K$ is a CM-field precisely when 
$\overline{-1}\not\in H$, and $K$ is a totally real field otherwise.

The special case that will occur in the remainder of this text,
is (with $i=\sqrt{-1}$) the field
\[K_d:=\QQ(\zeta_d-\zeta_d^{-1})=\QQ(2i\sin(2\pi/d))\subseteq\QQ(\zeta_d).\]
Some properties of $K_d$ are collected here.
\begin{lem}\label{fields}
    \begin{enumerate}
        \item $K_d$ is a CM-field precisely when $d\geq 3$.
        \item In case $4\nmid d$ one has $K_d=\QQ(\zeta_d)$.
        \item If $4|d$, then the CM-field $K_d\subset\QQ(\zeta_d)$ by the
        Galois correspondence relates to the subgroup
        of $(\ZZ/d\ZZ)^\times$ generated by
        $\overline{-1+d/2}$.
        \item The subgroup mentioned in (3) is trivial when
        $d=4$ and has order $2$ otherwise.
        \item In case $d=2^e\geq 8$, the Galois group of
        the CM-field $K_d$ over $\QQ$ is a cyclic group of order $2^{e-2}=d/4$.
        \item Every proper subfield of $K_{2^e}$ is totally real.
    \end{enumerate}
\end{lem}
\begin{proof}
  Assertion (1) follows from the discussion preceding the
  statement and the observation that $i\sin(2\pi/d)$ is real
  only when $d=1,2$.\\
  To prove (2) and (3) and (4), one observes that an automorphism of $\QQ(\zeta_d)$
  coming from $\overline{n}\in (\ZZ/d\ZZ)^\times$ fixes $\zeta_d-\zeta_d^{-1}$
  precisely when $\sin(2\pi/d)=\sin(2\pi n/d)$. The latter equality means
  that one of 
  \[\begin{array}{ccl} (\textrm{a}) &&1\equiv n\bmod d,\\ (\textrm{b})&&2\equiv d-2n\bmod 2d\end{array} \]
  holds. Case (a) means that we have the identity automorphism.
  Case~(b) can only occur when $2|d$, and then $n\equiv -1+d/2\bmod d$. However,
  as $n$ is supposed to be a unit modulo $d$, this also means that $d/2$ has
  to be even, i.e., $4|d$. In conclusion, the subgroup of $(\ZZ/d\ZZ)^\times$
  related to $\QQ(\zeta_d-\zeta_d^{-1})\subset\QQ(\zeta_d)$ in the 
  Galois correspondence, is trivial when $4\nmid d$. Hence in those
  situations the fields coincide, proving (2). Assertion (3) is clear from
  the same conclusion, and (4) is simply the observation that
  $-1+4/2=1$ and for $4|d$ such $d>4$ one has $\overline{-1+d/2}\in (\ZZ/d\ZZ)^\times$ has order $2$.

  To prove (5), from (3) and (4) it follows, with $d=2^e\geq 8$, that $\textrm{Gal}(K_d/\QQ)$ is isomorphic to $(\ZZ/d\ZZ)^\times/\langle \overline{-1+d/2}\rangle$.
  For completeness we include the following elementary reasoning. Consider
  the short exact sequence
  \[ 1\to \{\overline{n}\in (\ZZ/2^e\ZZ)^\times\;:\;n\equiv 1\bmod 4\}\hookrightarrow (\ZZ/2^e\ZZ)^\times\stackrel{\bmod 4}{\longrightarrow}
  (\ZZ/4\ZZ)^\times\to 1.\]
  The group on the left has order $2^{e-2}$, and $5\bmod 2^e$ is contained in it.
  Since $5^{2^{e-3}}\equiv 1+2^{e-1}\bmod 2^e$, the order of $5\bmod 2^e$
  equals $2^{e-2}$. Hence $\overline{5}$ generates the group on the left.
  From $\overline{-1+2^{e-1}}\not\in\langle \overline{5}\rangle$ one now
  concludes that the homomorphism $\langle \overline{5}\rangle \to
  (\ZZ/2^e\ZZ)^\times/\langle \overline{-1+2^{e-1}}\rangle$ given by
  $\overline{a}\mapsto \overline{a}\bmod \langle \overline{-1+2^{e-1}}\rangle$
  is injective. As both groups have the same cardinality $2^{e-2}$, this 
  shows (5).\\
  Assertion (6) is evident for $e\leq 2$. In the remaining situations we know
  that $\textrm{Gal}(K_{2^e}/\QQ)$ is a cyclic group of order $2^{e-2}$.
  This group contains a unique element of order $2$, and evidently this
  element is `complex conjugation'. Moreover, every nontrivial
  subgroup of the Galois group will contain this element. As a consequence,
  the subfield of $K_{2^e}$ corresponding to such a subgroup, is totally real.
  This completes the proof.
\end{proof}

\subsection{Chebyshev polynomials}
The \emph{Chebyshev polynomials of the first kind} $(\varphi_d)_{d\ge 0}$ are defined recursively by
\[
\varphi_0(x) = 2, \quad \varphi_1(x) = x, \quad \varphi_{d+1}(x) = x \varphi_d(x) - \varphi_{d-1}(x) \text{ for } d\ge 1.
\]
They satisfy the functional equation
\[
\varphi_d\Bigl(x + \frac{1}{x}\Bigr) = x^d + x^{-d}.
\]
These polynomials play an important role in the construction of certain hyperelliptic curves with CM Jacobians.
The basic idea is that a curve $\cC$ defined from an
equation $v^2=h(u)\varphi_d(u)$ where $h$ is a suitable
rational function, is covered by the curve $\cD$ defined from
$v^2=h(u+1/u)(u^d+u^{-d}$ via $u\mapsto x=u+1/u$.
Choices of $h$ such that $\textrm{Aut}(\cD)$ is large, then
may be exploited to obtain nontrivial elements of
$\textrm{End}(\textrm{Jac}(\cC))$.

\section{Proof of the main result}

The proof of Theorem~\ref{mainthm}
naturally splits in two distinct cases. 

\medskip
\noindent
\textbf{Case 1.} Suppose that $d=2^{e}$ for some integer $e\geq 1$.  

Starting more generally with {\em any}
positive $d\in 2\mathbb{Z}$, let $\cX_d$ be the curve defined by the affine equation
\[
y^{2} = x\bigl(x^{2d} + 1\bigr).
\]
In $\text{Aut}(\cX_d)$ one has the involution $\tau$ given by
\[
\tau(x,y) = \left( \frac{1}{x}, \, \frac{y}{x^{d+1}} \right)
\]
and the automorphism $\zeta$ defined by
\[ \zeta(x,y) = (\zeta_{4d}x,\zeta_{4d}y)\]
where $\zeta_{4d}$ denotes a primitive $4d$th root of unity.
Note that $\tau\zeta\tau=\zeta^{2d-1}$,
hence, for example, $\tau$ and $\zeta^2$ generate a dihedral group of
order $4d$. One notices that
$\zeta^{2d}$ is the hyperelliptic involution of the curve $\cX_d$.
For any $\sigma\in\textrm{Aut}(\cX_d)$
we denote by $[\sigma]$ the corresponding automorphism of the
jacobian $\textrm{Jac}(\cX_d)$.
Then $[\zeta^{2d}]=[-1]$
and $[\zeta]+[\tau\zeta\tau]=[\zeta]-[\zeta^{-1}]$ commutes with $[\tau]$.

We claim that the quotient curve $\cC_d = \cX_d / \langle \tau \rangle$ is described by the equation
\[
v^{2} = (u + 2)\,\varphi_{d}(u).
\]
Indeed, 
\[
u = x + \frac{1}{x}, 
\qquad 
v = y\,\frac{x + 1}{x^{1 + d/2}}
\]
are $\tau$-invariant functions, and
\[
v^{2} = (x^{2d+1}+x)(x+2+x^{-1})x^{-d-1}=(u + 2)\,\varphi_{d}(u).
\]
Denoting by $k(x,y)$ the function field of $\cX_d$ over any field $k$
of characteristic not dividing $2d$, 
one obtains a diagram of field extensions

\[\begin{array}{ccccc}
k(x,y) & \stackrel{\scriptstyle{2}}{\supset} & 
k(x,y)^{\langle \tau \rangle} & \supset & k(u,v)\\
\cup\scriptstyle{2} && \cup && \cup\scriptstyle{2} \\
k(x) & \stackrel{\supset}{\scriptstyle{2}} & k(u) & = & k(u)
\end{array}\]
from which one concludes $k(x,y)^{\langle \tau \rangle}=k(u,v)$. The claim is an easy
consequence of this.

The degree $2$ map $\cX_d\to\cC_d$
induces a homomorphism $\textrm{Jac}(\cC_d)\to\textrm{Jac}(\cX_d)$ with
finite kernel.
Following the argument in \cite{TTV}, to prove our main result it suffices to show that the endomorphism
\[
[ \zeta] - [ \zeta]^{-1}
\]
restricts to an endomorphism
of the image of $\textrm{Jac}(\cC_d)$
in $\textrm{Jac}(\cX_d)$. This will be the case if the associated
tangent map on the tangent space at the origin of $\textrm{Jac}(\cX_d)$ preserves the subspace on which the involution $\tau$ acts trivially. Equivalently, it is enough to verify that the element
\[
 \zeta - \zeta^{-1} \in \mathbb{Z}[\operatorname{Aut}(\cX_d)]
\]
stabilizes the subspace of $\tau$-invariant differentials in
\[
H^0(\cX_d, \Omega^1_{\cX}).
\]
The observation that $\tau$ and $\zeta+\tau\zeta\tau$ commute in $\mathbb{Z}[\operatorname{Aut}(\cX_d)]$
implies that this is indeed the case.

Explicitly, a basis of the space of regular differentials on $\cX_d$ is
\[
\{ \omega_j := \frac{x^{j-1} dx}{y}, \quad 1 \le j \le d \},
\]
and a basis for the subspace of $\tau$-invariant differentials is
\[
\{ \omega_j - \omega_{d-j+1}, \quad 1 \le j \le d/2 \}.
\]

The equality
\[
(\zeta-\zeta^{-1})^*(\omega_j-\omega_{d-j+1}) 
= (\zeta_{4d}^{2j-1}-\zeta_{4d}^{-(2j-1)})(\omega_j-\omega_{d-j+1})
\]
shows
\[\mathbb{Q}(\zeta_{4d}-\zeta_{4d}^{-1})\hookrightarrow \textrm{End}^0(\textrm{Jac}(\cC_d) \]
and yields the associated action on the
tangent space at the origin.
From Lemma~\ref{fields} and the
assumption $4|d$ one deduces
$[\mathbb{Q}(\zeta_{4d}-\zeta_{4d}^{-1}):\mathbb{Q}]=\phi(4d)/2$.
Since $\dim\,\textrm{Jac}(\cC_d)=d/2$,
the given embedding of rings makes
$\textrm{Jac}(\cC_d)$ a CM Jacobian
precisely when  $\phi(4d)=2d$,
or equivalently, when $d=2^e$ for some
$e\geq 1$.

Assuming $d=2^e\geq 2$, Lemma~\ref{fields}(6) implies that
the corresponding CM-type is primitive.
 Therefore, $\textrm{Jac}(\cC_{2^e})$ is simple.

 Note that the action of
 $(\zeta-\zeta^{-1})^*$ on $\tau$-invariant differentials as described above,
 shows that the CM-type here is described as the subset
 \[ \{1, 3, 5, \ldots, 2^e-1\}\subset(\mathbb{Z}/2^{e+2}\mathbb{Z})^\times/\langle 2^{e+1}-1\rangle\cong\textrm{Gal}(\mathbb{Q}(\zeta_{2^{e+2}}-\zeta_{2^{e+2}}^{-1})/\mathbb{Q}).\]

\medskip
\noindent
\textbf{Case 2.} 
Suppose now that $d  \geq 3$ is odd.  
Let $\cD_{2d}$ be the curve defined by
\[
\cD_{2d} : y^{2} = x^{2d} + 1
\]
and as before, denote by $\cC_d$ the curve defined by
\[
\cC_d\colon  y^{2} = (x+2)\varphi_{d}(x).
\]
As in the proof of Theorem~4.1 in \cite{TT},
let $\sigma$ be the involution on $\cD_{2d}$ defined by $\sigma(x,y)=(1/x,y/x^d)$.
The quotient of $\cD_{2d}$ by $\sigma$ is the curve
$\cC_d$. Indeed, similar to the argument
in Case~1 above, the functions $u=x+x^{-1}$ and $v=y(1+x)x^{-(d+1)/2}$
generate the field of functions invariant under $\sigma$,
and one computes
\[ v^2=y^2\cdot x^{-(p+1)}\cdot(x+1)^2=(x^p+x^{-p})(x+2+x^{-1})=(u+2)\varphi_p(u).\]
Define the subgroup $G \subset \mbox{Aut}(\cD_{2d})$ generated by
\[
\zeta: (x,y) \mapsto (\zeta_{2d} x, \,  y),
\]
where $\zeta_{2d}$ is a primitive $2d$-th root of unity. Let $[\zeta]$ denote the induced automorphism on the Jacobian $\textrm{Jac}(\cD_{2d})$.

 A basis of the space of regular differentials on $\cD_{2d}$ is
\[
\{ \omega_j := \frac{x^{j-1} dx}{y}, \quad 1 \le j \le d-1 \},
\]
and the $\sigma$-invariant differentials admit as a basis
\[
\{ \omega_j - \omega_{d-j}, \quad 1 \le j \le (d-1)/2 \}.
\]

Analogous to Case~1, this leads to
\[
\QQ(\zeta_{2d}-\zeta_{2d}^{-1}) \subset \mbox{End}^0(\textrm{Jac}(\cC_d)),
\]
using
\[
(\zeta-\zeta^{-1})^{*}(\omega_j-\omega_{d-j}) 
= (\zeta_{2d}^{j}-\zeta_{2d}^{-j})(\omega_j-\omega_{d-j}).
\]

 Since $d$ is odd, one has (using, e.g., Lemma~\ref{fields}(2)) that
$\mathbb{Q}(\zeta_{2d}-\zeta_{2d}^{-1})=\mathbb{Q}(\zeta_{d})$,
and the embedding of rings described above
makes $\textrm{Jac}(\cC_d)$ a CM Jacobian precisely when
$\phi(d)=d-1$, or equivalently, when $d=p$ is a prime number (here: odd).

For such $d=p\geq 3$ prime, the corresponding CM-type is
\[
    S = \{\,1,2,\dots,(p-1)/2\,\} \subset (\mathbb{Z}/p\mathbb{Z})^\times
    \;\cong\; \Gal(\mathbb{Q}(\zeta_p)/\mathbb{Q}),
\]
where an integer $a$ denotes the automorphism given by $\zeta_p \mapsto \zeta_p^a$.

A standard criterion (phrased here for a CM-field $K$ that is Galois over $\mathbb{Q}$) states: the CM-type $(K,S)$ is \emph{induced} from a CM-type of a proper CM subfield $F\subset K$ if and only if $S$ is a union $\cup  H\sigma_j$ of cosets of the nontrivial subgroup $H=\Gal(K/F)\subset G$.
In other words, $H$ acts nontrivially on the set $S$.

In our situation, with \(K = \mathbb{Q}(\zeta_p)\), 
it follows from the following classical result (see, for instance, \cite[page~24]{Lang}, which uses
\cite[Prop.~26]{ST}) 
that the associated CM-type is primitive, and consequently $\textrm{Jac}(\cC_p)$ is simple.

\begin{pro}
Let \(p\ge 3\) be an odd prime and \(K=\mathbb{Q}(\zeta_p)\) the $p$-th cyclotomic field.

For \(1\le i\le n\), \(\varphi_i\in\Gal(K/\mathbb{Q})\) denotes the automorphism determined by
\(\varphi_i(\zeta_p)=\zeta_p^i\).
Then the CM-type
\[
S=\{\varphi_1,\varphi_2,\dots,\varphi_{(p-1)/2}\}
\]
is primitive.
\end{pro}

\begin{proof}
For convenience we include the well-known argument. Write \(G=\Gal(K/\mathbb{Q})\cong(\mathbb{Z}/p\mathbb{Z})^\times\).
Suppose \(\sigma\in G\) leaves the set \(S\) stable. Identifying \(S= \{1,2,\dots,(p-1)/2\}\subset(\mathbb{Z}/p\mathbb{Z})^\times\) and $\sigma\colon \zeta_p\mapsto\zeta_p^a$,
this means that 
\[
\{a\cdot 1,\; a\cdot 2,\; \dots,\; a\cdot (p-1)/2\}
\equiv
\{1,\;2,\;\dots,\;(p-1)/2\}
\pmod p.
\]
Hence also
\[
a\sum_{i=1}^{(p-1)/2} i \equiv \sum_{i=1}^{(p-1)/2} i \pmod p.
\]
Since
\[
\sum_{i=1}^{(p-1)/2} i  
= (p^2-1)/8
\]
is a unit modulo \(p\), it follows that \(a\equiv 1\pmod p\), hence \(\sigma\) is the identity.

Thus the stabilizer of \(S\) in \(G\) is trivial, so \(S\) is not a union of cosets of any nontrivial subgroup of \(G\).
Therefore the CM-type \((K,S)\) is primitive.
\end{proof}

\begin{remark} \rm{
Writing (for $p$ an odd prime number)
$\cD_p\colon y^2=x^p+1$, the well-known CM-type on $\textrm{Jac}(\cD_p)$
and the one presented above on
$\textrm{Jac}(\cD_p)$ coincide.
As a consequence, the two Jacobians are isogenous.
An alternative way to see this, is as follows (notations as before).   
Consider the involution $\tau(x,y)=(-x,y)$ on $\cD_{2d}$.  
The quotient $\cD_{2d}/\langle\tau\rangle$ is $\cD_d\colon y^2=x^d+1$, 
and by \cite[Theorem~5]{Pau} one obtains 
$\textrm{Jac}(\cD_{2d})\sim \textrm{Jac}(\cD_d)^2$.  
Considering (now for $d$ odd) the involution $\sigma(x,y)=(1/x,y/x^d)$ on $\cD_{2d}$, 
whose quotient is the curve $\cC_d$, the same theorem gives 
$\textrm{Jac}(\cD_{2d})\sim \textrm{Jac}(\cC_d)^2$.  
Combining these relations yields
\[
  \textrm{Jac}(\cD_d)^2 \sim \textrm{Jac}(\cC_d)^2,
  \qquad\text{hence}\qquad
  \textrm{Jac}(\cD_d)\sim \textrm{Jac}(\cC_d).
\]
}
\end{remark}

 \section*{Acknowledgement}
{The first author was partially supported by CNPq grant No. 302774/2025–
4 and Fapesp grant No. 2024/00923–6.
}

 \end{document}